\documentclass{article}
\usepackage[T1]{fontenc}
\usepackage{wrapfig}

\usepackage[english]{babel}
\usepackage{booktabs}
\usepackage{multirow}

\usepackage{tipa}
\usepackage[letterpaper,top=2cm,bottom=2cm,left=3cm,right=3cm,marginparwidth=1.75cm]{geometry}

\usepackage{amsmath, amsthm, amssymb}
\usepackage{graphicx}
\usepackage[colorlinks=true, allcolors=blue]{hyperref}
\usepackage{comment}

\newcommand{\CC}{\mathbb{C}}
\newcommand{\RR}{\mathbb{R}}
\newcommand{\QQ}{\mathbb{Q}}
\newcommand{\ZZ}{\mathbb{Z}}

\newcommand{\kk}{\Bbbk}

\newcommand{\FF}{\mathbb{F}}

\newcommand{\sslash}{\mathord{/\mkern-6mu/}}

\newcommand{\Aut}{\operatorname{Aut}}

\newtheorem{theorem}{Theorem}[section]

\newtheorem*{theorem*}{Theorem}
\newtheorem*{proposition*}{Proposition}
\newtheorem*{conjecture*}{Conjecture}

\theoremstyle{definition}
\newtheorem{definition}[theorem]{Definition}
\newtheorem{example}[theorem]{Example}

\title{If you can distinguish, you can express: \\ Galois theory, Stone--Weierstrass, machine learning, and linguistics}
\author{Ben Blum-Smith, Claudia Brugman, Thomas Conners, and Soledad Villar}
\date{}

\begin{document}
\maketitle

\begin{abstract}
This essay develops a parallel between the Fundamental Theorem of Galois Theory and the Stone--Weierstrass theorem: both can be viewed as assertions that tie the distinguishing power of a class of objects to their expressive power. We provide an elementary theorem connecting the relevant notions of ``distinguishing power''. We also discuss machine learning and data science contexts in which these theorems, and more generally the theme of links between distinguishing power and expressive power, appear. Finally, we discuss the same theme in the context of linguistics, where it appears as a foundational principle, and illustrate it with several examples.
\end{abstract}

\section{Introduction}\label{sec:intro}

This article identifies a theme common to certain important results in algebra and analysis, which also manifests as a fundamental principle in linguistics. It is a collaboration between a pair of mathematicians and a pair of linguists. It begins with an analogy: Two seemingly unrelated tools from different areas of mathematics are used in an applied domain to obtain results of a very similar form. The tools are the Stone--Weierstrass theorem and the Fundamental Theorem of Galois Theory; the applied domain is machine learning. In this brief introduction, we outline the goals and structure of the article.

Our purpose is (not to present new theorems, but) to reinterpret existing theorems to draw out a common theme.  The main idea is this: both the Stone--Weierstrass Theorem and the Fundamental Theorem of Galois Theory can be framed as statements that tie the expressive power of a set of objects to their distinguishing power. This idea is elaborated in Section~\ref{sec:main-idea}. This framing wraps neatly around the usual formulation of the Stone--Weierstrass theorem (the reader may already see how), but it requires a little more reinterpretation of the Fundamental Theorem of Galois Theory, and the main work of Section~\ref{sec:main-idea} is to carry out that reinterpretation. This reinterpretive work also draws into the same framing a third classical theorem, Rosenlicht's Theorem (from the theory of algebraic groups).

To articulate the main analogy, we operationalize the notions of ``expressive power'' and ``distinguishing power'' in different ways in the Galois vs. Stone--Weierstrass contexts. In Section~\ref{sec.formal}, we tighten the analogy with a recent theorem that ties together the different notions of ``distinguishing'' in the two contexts.

Having identified the theme of (of connections between distinction and expression) common to these classical results, we then look around for other instantiations. In Section~\ref{sec:in-data-science}, we collect various instances woven throughout the data science and machine learning literature. In Section~\ref{sec:linguistics}, we turn to the field of linguistics, where it turns out that the idea that distinction gives rise to expression is a foundational principle; we illustrate this principle with several examples. Meanwhile, in Section~\ref{sec:non-ex}, we discuss an important caveat: while we hope the rest of the paper shows that links between distinction and expression are common throughout mathematics and elsewhere, not every statement that one might hope for of the form ``distinguish $\Rightarrow$ express''  actually holds. We give a classical example from invariant theory where a relevant notion of distinguishing power is insufficient for a relevant notion of expressive power, and a more recent example from the machine learning literature.

Both the Stone--Weierstrass Theorem and the Fundamental Theorem of Galois Theory are bidirectional implications. Still, for the purposes of this essay, we focus on the implication ``if you can distinguish, then you can express'' as the main story (hence the title), and we regard any converse statements as a side story. Our motivation for this is twofold:
\begin{itemize}
    \item In various of the mathematical contexts discussed below, the implication {\em distinguish $\Rightarrow$ express} is the more substantive one. In particular, in the Stone--Weierstrass case, it is the main content of the theorem. In Section~\ref{sec:non-ex} on counterexamples, it is this implication that fails.
    \item In the context of linguistics, the main implication shows up as a foundational tenet as discussed in Section~\ref{sec:linguistics}, but exploring the converse implication would have necessitated an investigation well beyond our scope regarding the outer limits of the meaning of the word ``express'' in that context.
\end{itemize}

\section{What does the Fundamental Theorem of Galois Theory have to do with the Stone--Weierstrass theorem?}\label{sec:main-idea}

In this section, we show that both the Stone--Weierstrass theorem and the Fundamental Theorem of Galois Theory can be viewed as articulating the following principle:

\begin{center}
    {\em If you can distinguish, then you can express (and conversely).}
\end{center}

This is more transparent for the Stone--Weierstrass theorem (e.g., \cite[Chapter~9, Theorem~34]{royden}). One articulation of the theorem is as follows:

\begin{theorem*}[Stone--Weierstrass]
    Let $X$ be a compact Hausdorff topological space. Let $C(X,\RR)$ be the Banach space of continuous real-valued functions on $X$, with the sup norm; it is a Banach algebra under pointwise multiplication of functions. Let $S\subset C(X,\RR)$ be a subset. Then the algebra $\RR[S]$ generated over $\RR$ by $S$ is dense in $C(X,\RR)$ if and only if the elements of $S$ separate the points of $X$.
\end{theorem*}

The ``only-if'' direction is an immediate consequence of Urysohn's lemma. Compact Hausdorff spaces are normal, so Urysohn's lemma implies that any two points $x_1,x_2\in X$ are separated by {\em some} continuous function, i.e., there exists some $f\in C(X,\RR)$ with $\varepsilon:=|f(x_1)-f(x_2)| > 0$. If $x_1,x_2$ are not separated by any element of $S$, then they are not separated by any element of $\RR[S]$, and thus no element of $\RR[S]$ can be closer than $\varepsilon / 2$ to $f$ in the sup norm.

The more substantive direction, for which we will not outline the proof, is the ``if'' direction. It asserts that if the set of functions $S$ is rich enough to separate any pair of points in $X$, then given any accuracy level $\varepsilon > 0$, any function $f\in C(X,\RR)$ can be uniformly $\varepsilon$-approximated by some $f'\in \RR[S]$.

Said differently, if the elements of $S$ can distinguish the points of $X$, then one can express any continuous function in terms of these elements---where in this context,  {\em express} means to create an $\varepsilon$-approximation to the desired function using the elements of $S$ as ingredients, mixed together via the algebra operations of linear combination and multiplication. The easier, only-if direction of the theorem articulates the converse: if $S$ can express the elements of $C(X,\RR)$ in the above sense, then they can distinguish the points of $X$.

The sense in which the Fundamental Theorem of Galois Theory (e.g., \cite[pp.~239]{jacobson}) articulates the same theme (i.e., {\em distinguish $\Rightarrow$ express}) is perhaps more subtle, at least in the usual formulation of the theorem, which is as follows:

\begin{theorem*}[Fundamental Theorem of Galois Theory]
    Let $L/k$ be a finite, normal, separable field extension. Then the group $G:=\Aut_k(L)$ of $k$-algebra automorphisms of $L$ is finite of order $[L:k]$, and the set of fixed points $L^G$ for the action of $G$ on $L$ is precisely $k$. Furthermore, the subfields $k\subset K\subset L$ of the extension are in inclusion-reversing bijection with the subgroups $H$ of $G$. The bijection is given by the inverse maps
    \[
    L\supset K \mapsto \Aut_K(L)\subset G
    \]
    from subfields to subgroups, and 
    \[
    G\supset H \mapsto L^H\subset L
    \]
    from subgroups to subfields.\footnote{The theorem is usually stated with the additional information that the bijection between subgroups and subfields sends index of subgroups in $G$ to degree of field extensions over $k$, and normal subgroups of $G$ to normal extensions of $k$; we relegate this additional information to the present footnote as it plays no role in what follows.}
\end{theorem*}

Where are the distinguishing and expressing? 

Let $H$ be some subgroup of $G$. Suppose we can find some elements $f_1,\dots,f_s$ of $L$ with the following pair of properties:
\begin{enumerate}
    \item $h f_j = f_j$ for every $j=1,\dots, s$ and every $h\in H$.\label{item:invariant}
    \item For $g\in G\setminus H$, there exists some $j^\star\in \{1,\dots, s\}$ such that $g f_{j^\star} \neq f_{j^\star}$.\label{item:only-invariant-under}
\end{enumerate}
Then consider the field $k(f_1,\dots,f_s)$ generated by the $f_j$'s over $k$. Property \ref{item:invariant} tells us that all the elements of $H$ are $k(f_1,\dots,f_s)$-automorphisms of $L$; 
Property \ref{item:only-invariant-under} tells us that the only elements of $G$ that are $k(f_1,\dots,f_s)$-automorphisms of $L$ are those that lie in $H$. Thus, the map $K\mapsto \Aut_K(L)$ from subextensions to subgroups maps $k(f_1,\dots,f_s)$ to $H$. Since the Fundamental Theorem tells us that the composition of this map with the map $H\mapsto L^H$ is the identity, we have $k(f_1,\dots,f_s) = L^H$.

The previous paragraph extracts a consequence of the Fundamental Theorem of Galois Theory whose hypothesis can be summarized as follows: the elements $f_1,\dots,f_s$ distinguish the elements of $H$ from the rest of the elements of $G$ (by being simultaneously invariant [only] under the elements of $H$). The conclusion can be summarized: any $H$-invariant element of $L$ can be expressed in terms of $f_1,\dots,f_s$. In this context, {\em express} means to write the desired element of $L$ using $f_1,\dots,f_s$ as ingredients, mixed together via the field operations of $k$-linear combination, multiplication and division. Thus, the theorem tells us that {\em distinguish $\Rightarrow$ express}.

The direction {\em express $\Rightarrow$ distinguish} follows from the theorem as well. If $f_1,\dots,f_s$ generate $L^H$ over $k$, then the map $H\mapsto L^H$ sends $H$ to $k(f_1,\dots,f_s)$; the composition with the inverse map $K\mapsto \Aut_K(L)$ is the identity, so the image of $k(f_1,\dots,f_s)$ under this map is $H$. In other words,  properties~\ref{item:invariant} and \ref{item:only-invariant-under} above are satisfied, i.e., the (only) $k$-automorphisms of $L$ that fix $f_1,\dots,f_s$ are precisely those in $H$.

Once things are formulated in this way, it surfaces that Rosenlicht's theorem, from the theory of algebraic groups, also fits the theme:

\begin{theorem}[Rosenlicht's theorem \cite{rosenlicht1956some}]
    Let $G$ be an algebraic group, acting regularly on an irreducible algebraic variety $V$ over a field $k$, with algebraic closure $\overline k$. Let $k(V)$ be the function field of $V$, and $k(V)^G$ the field of rational $G$-invariants. Then the elements of $k(V)^G$ separate the orbits of $G$ on the $\overline k$-points of $V$ away from a proper Zariski-closed subset. 
    
    Conversely, if $f_1,\dots,f_s\in k(V)^G$ separate orbits of $G$ on the $\overline k$-points of $V$ away from a proper Zariski-closed subset, and if, furthermore, the field extension $k(V)/k(f_1,\dots,f_s)$ is separable, then $f_1,\dots,f_s$ generate $k(V)^G$ over $k$.
    
\end{theorem}

We have paraphrased Rosenlicht's original formulation here to emphasize the algebraic (as opposed to geometric) content, and to make it easier for the modern reader to read. (Rosenlicht's paper, written in 1956, is in the language of the {\em Weil foundations} of algebraic geometry \cite{weil1946foundations}, which were superseded by Grothendieck's foundational work shortly thereafter \cite{grothendieck1960elements}.)

This theorem is yet another instance of {\em distinguish $\Rightarrow$ express}. Its interpretation of {\em express} is the same as the Galois theorem, while its {\em distinguish} comes very close to that of Stone--Weierstrass. To elaborate---if $k$ is of characteristic zero (for example if $k=\RR$ or $\CC$), then any extension is separable, and the theorem states that elements $f_1,\dots,f_s$ of the field $k(V)^G$ of rational invariants generically separate $G$-orbits over the algebraic closure (distinguish), if and only if they generate all rational invariants with respect to the field operations (express).

\section{A precise connection between the two notions of \emph{distinguish} in Stone--Weierstrass and Galois Theory}
\label{sec.formal}

The previous section explains how both the Stone--Weierstrass theorem and the Fundamental Theorem of Galois Theory tell a story of the form \emph{distinguish $\Rightarrow$ express}, with different notions for \emph{distinguishing} and \emph{expressing} in each case. This connection has the character of a formal analogy. While we trust the reader sees a close connection between the two notions of {\em expressing} (using the tools of $\{\RR\text{-linear combination},\times,\varepsilon\text{-approximation}\}$ in the Stone--Weierstrass case, as compared with $\{k\text{-linear combination},\times,\div\}$ in the Galois case), the two notions of {\em distinguishing} are not as prima facie similar. In this section, we tighten the analogy with an elementary theorem that directly relates the two. Here we reformulate Theorem 3.1 of \cite{blum2024galois} to place it in this conceptual framework.

\begin{definition}[generically Stone--Weierstrass-distinguishing set of functions for a group $G$]
Let $X$ be a topological measure space with an action by a group $G$, and consider a set of $G$-invariant functions $f_1,\ldots f_r$ from $X$ to some abelian group $\mathbb F$. We say that  $f_1,\dots,f_r$ are  {\em generically SW-distinguishing for $G$} (or what is more commonly known as {\em generically separating}) if there exists a closed, measure zero, $G$-stable subset $B\subset X$ (the ``bad set'') such that  $x_1,x_2\in X\setminus B$ must lie in the same orbit of $G$ if $f_j(x_1)=f_j(x_2)$ for all $j=1,\ldots,r$.     
\end{definition}

Now let $H\subset G$ be a subgroup of finite index. The connection we draw in this section shows that we can extend a generically SW-distinguishing set for $G$ to a generically SW-distinguishing set for $H$ by adding a set of functions $f^\star_1,\dots,f^\star_s:X\rightarrow \FF$ that distinguish $G$ from $H$ in a Galois sense we define below.
To do so, we consider a class of functions $\mathcal{F}$ from $X$ to $\FF$ that is an abelian group under pointwise addition, is closed under the natural action of $G$ by $( g  f)(x):=f( g ^{-1}x)$ for $x\in X,  g \in G$, and such that the nonzero elements of $\mathcal{F}$ have closed, measure-zero vanishing sets. For many examples, $\FF$ is $\RR$ or $\CC$, and  $\mathcal{F}$ could be the class of polynomial functions, or, if $X$ is a $\CC$- or $\RR$-vector space (or variety, or analytic manifold) and $G$ a linear group (or algebraic group, or Lie group acting by analytic morphisms), then $\mathcal{F}$ could be the class of analytic functions; and other classes that come up in machine learning also have this property (for example certain classes of functions related to multi-layer perceptrons with a prespecified analytic activation function, such as a sigmoid like the logistic function \cite{berkson1944application}). We remark that we do not explicitly impose any restrictions on the way $G$'s action on $X$ interacts with the latter's topological or measure structure. The necessary restrictions are instead hidden in the assumptions on $\mathcal{F}$, in the sense that if $G$'s action on $X$ does not cooperate with the topological and measure structures, then classes $\mathcal{F}$ satisfying the hypotheses may be hard to find. 

\begin{definition}[Galois-distinguishing set of functions for a subgroup $H\subset G$]
   Let $X,\mathbb F, G,H,\mathcal F$ be as above. We say that the functions $f^\star_1,\dots,f^\star_s:X\rightarrow\FF$ Galois-distinguish $H$ from $G$ if they are $H$-invariant functions in $\mathcal{F}$ such that the only group elements $g\in G$ that fix all of them belong to $H$. That is, for $ g \in G$ we have
    \[
     g  f^\star_j = f^\star_j\text{ for all }j\in 1,\ldots, s\quad \Rightarrow \quad  g \in H.
    \]
\end{definition}

\begin{theorem}[Reframing of Theorem 3.1 in \cite{blum2024galois}]\label{thm:galois-for-ML}
    Let $X, \mathbb F, G, H,\mathcal{F}$ be as above.  Suppose $f_1,\dots,f_r: X\rightarrow\FF$ are $G$-invariant functions that are generically SW-distinguishing for $G$. If $f^\star_1,\dots,f^\star_s:X\rightarrow\FF$ are $H$-invariant functions belonging to $\mathcal{F}$ that Galois-distinguish $H$ from $G$,
    then $f_1,\dots,f_r,f^\star_1,\dots,f^\star_s$ are generically SW-distinguishing for $H$.
\end{theorem}

\begin{proof}
    Since $f^\star_1,\dots,f^\star_s$ are $H$-invariant,  the stabilizer $G_j$ of each $f_j$ in $G$ contains $H$. The  hypothesis on the $f^\star_j$ imply $\bigcap_j G_j = H$.
    Since $[G:H]<\infty$, each $G_j$ has finite index in $H$. Thus, there are only finitely many functions of the form $g f^\star_j$, $g\in G$, namely one for each pair $(j, g G_j)$ consisting of $j\in [s]$ and a left coset of the stabilizer $G_j$. They all belong to $\mathcal{F}$, because $\mathcal{F}$ is $G$-stable. Thus, the finitely many functions
    \[
    g f^\star_j - f^\star_j, \; j\in[s], \;g\notin G_j
    \]
    all belong to $\mathcal{F}$ as well (because it is an abelian group). Furthermore, they are all nonzero because $g\notin G_j$ for each one. So by the hypothesis on $\mathcal{F}$, they all have closed, measure zero vanishing sets. Let 
    \begin{align} \label{eq:explicit-bad-set} 
    B=\bigcup_{j\in[s], g\in G\setminus G_j}\{x\in X:(g f^\star_j-f^\star_j)(x)=0\}
    \end{align}    
    be the union of these; by the above, this is a union of only finitely many distinct closed, measure zero sets, thus it is closed and measure zero. It is $H$-stable by construction.
    
    Meanwhile, because $f_1,\dots,f_r$ are SW-distinguishing for $G$, we know that there exists another closed, measure zero, $G$-stable set $B'$ on the complement of which any two distinct $G$-orbits are distinguished by some $f_j$.

    Then $B\cup B'$ is closed, measure zero, and $H$-stable (because $B'$ is $G$-stable and $B$ is $H$-stable). We claim that any $x_1,x_2 \in X\setminus(B\cup B')$ lying in distinct $H$-orbits are distinguished either by some $f_j$ or by some $f^\star_j$. Indeed, if $x_1,x_2$ lie in distinct $G$-orbits, then they are distinguished by some $f_j$; while if they lie in the same $G$-orbit but distinct $H$-orbits, then there exists $g\in G\setminus H$ with $g x_1 = x_2$. In the latter case, there exists $j\in [s]$ with $g \notin G_j$ because $H=\bigcap_j G_j$, and then 
    \begin{align*}
    f^\star_j(x_1) - f^\star_j (x_2) &= f^\star_j(g^{-1}x_2) - f^\star_j(x_2)\\
    &= (g f^\star_j-f^\star_j)(x_2)\\
    &\neq 0,
    \end{align*}
    where the final inequality is because $x_2\notin B$ (and $B$ contains the vanishing set of $g f^\star_j - f^\star_j$ by definition). So $x_1,x_2$ are distinguished by $f^\star_j$.
\end{proof}

\section{{\em Distinguish $\Rightarrow$ express} in data science and machine learning}\label{sec:in-data-science}

Because distinguishing data points, and expressing quantities of interest, are fundamental to so many mathematical tasks, the theme discussed above comes up in a wide variety of applications in machine learning and data science. We discuss a sample of these; it is far from comprehensive.

\subsection{Stone--Weierstrass in machine learning} \label{subseq.sw.ml}
The basic example of a machine learning architecture is a {\em feedforward neural network}. The components that comprise this architecture are a data (input) space $\mathcal X = \RR^{n_0}$, an output space $\mathcal Y = \RR^{n_{s+1}}$, a number of ``hidden layers'' $L_1=\RR^{n_1},\dots,L_s=\RR^{n_s}$, and a ``pointwise nonlinearity'' $\sigma$, also called an ``activation function'', which is a fixed function $\RR\rightarrow \RR$ that gets applied to each entry of a vector. The number of hidden layers $s$ is the ``depth" of the network, and the maximum of their individual dimensions $n_1,\dots,n_s$ is the ``width". The family of maps $\mathcal X \rightarrow\mathcal Y$ parametrized by this architecture consists of compositions
\begin{equation} \label{eq.mlp}
M_{s+1} \circ \sigma \circ M_s\circ\sigma \circ \dots\circ \sigma\circ M_1,
\end{equation}
where each $M_j$ is a linear or affine map $L_{j-1}\rightarrow L_j$ with $L_0=\mathcal X$ and $L_{s+1} = \mathcal Y$, and $\sigma$ is applied entrywise so as to make it a function $L_j\rightarrow L_j$ for each $j$. The parameters are the linear or affine maps $M_j$. If each $M_j$ is allowed to be a completely arbitrary affine map $L_{j-1}\rightarrow L_j$, then the network is said to be {\em fully connected}---and this designation is also used if the last map $M_{s+1}$ is constrained to be linear---but it is common in practice, and necessary for equivariant architectures, to constrain the available $M_j$'s further. A fully connected feedforward neural network is also called a {\em multilayer perceptron}, or {\em MLP}. 

The Stone--Weierstrass theorem is a tool that has been extensively used to prove universality results in machine learning. For example, the classical paper of Hornik, Stichcombe, and White from 1989 shows that certain classes of MLPs are universal approximators \cite{hornik1989multilayer}. Their approach is to define a class of MLPs (which turns out to be an algebra) and show that it separates inputs. Then the Stone--Weierstrass theorem guarantees expressivity of the class of functions. Similar results for different neural network models are described in the 1999 survey by Pinkus \cite{pinkus1999approximation}.

A formulation of the classic universality theorem based on the Stone-Weierstrass theorem is that an MLP with a single hidden layer of unbounded width is universally approximating. More precisely \cite[Theorem~3.1]{pinkus1999approximation}:

\begin{theorem*}[Universal approximation for MLPs]
With the above notation, suppose $s=1$ (so there is one hidden layer), $n_2=1$ (so that $L_2=\mathcal Y = \RR$, i.e., the output is a single number), and $\sigma$ is any continuous function that is not a polynomial. Let $f:\mathcal X \rightarrow \mathcal Y=R$ be a continuous function we aim to approximate, $C\subset \mathcal X$ a compact set, and $\varepsilon > 0$ an error threshold. Then there exists a natural number $n_1$, an affine map $M_1:\mathcal X \rightarrow L_1 = \RR^{n_1}$, and a linear functional $M_2:L_1\rightarrow \mathcal{Y}=\RR$, so that the composition $M_2\circ\sigma\circ M_1$ uniformly $\varepsilon$-approximates $f$ on $C$.
\end{theorem*}

These tools have later been used to prove analogous universality results for equivariant machine learning models, that is, machine learning models that respect symmetries. In equivariant machine learning one typically considers classes of functions $\mathcal F = \{f_\theta : X\to Y, \, \theta \in \mathbb R^d\}$ so that for every choice of parameters $\theta$, the corresponding function $f_\theta$ obeys a prescribed symmetry which is expressed as an invariance or equivariance with respect to a group action \cite{cohen2021equivariant}. Approaches to proving universality of these models include Stone--Weierstrass arguments \cite{dymuniversality, blum2024galois} (e.g., in point clouds), and arguments based on averaging already universal models along group orbits \cite{yarotsky2022universal, puny2021frame}, for example by having a hidden layer with an arbitrary number of channels where the group acts regularly \cite{ravanbakhsh2020universal}. However, attaining universality theorems for equivariant models can be trickier than for non-equivariant ones.

Graph neural networks (GNNs) \cite{duvenaud2015convolutional} are a great example of equivariant machine learning models with interesting expressivity properties. GNNs are defined on graphs $G=(V,E,X)$ where $V=[n]$ denotes the set of nodes, $E\in \mathbb R^{n\times n}$ denotes a matrix, often taken to be symmetric, giving weights on the edges, and $X\in \mathbb R^{n\times d}$ are the node features. The typical tasks GNNs perform are learning graph-level functions $f:G\to \mathbb R^k$ or node embeddings $f:G\to \mathbb R^{n\times k}$. Both of these learning tasks exhibit a symmetry due to the order of the nodes not being an intrinsic property of the graph itself (known in physics as a passive symmetry \cite{villartowards}), namely $(\pi V, \pi E \pi^\top, \pi X) = (V, E, X)$ for all permutations $\pi \in S_n$. Graph-level learning tasks are invariant to this group action whereas node-level learning tasks are equivariant. 

If a class of invariant functions under the symmetric group acting by conjugation separates orbits, then it can distinguish every pair of non-isomorphic graphs. Since the graph isomorphism problem is computationally intractable with current techniques (although there is a recent quasi-polynomial-time algorithm \cite{babai2016graph}), it follows from the \emph{express $\Rightarrow$ distinguish} direction of Stone--Weierstrass that standard implementations of graph neural networks are not universal \cite{chen2019equivalence}, except in cases where the complexity of the architecture is allowed to grow super-polynomially with the size of the input \cite{keriven2019universal, maron2019universality}. In fact, there is a large literature that studies the expressive power of graph neural networks in connection with graph-isomorphism tests \cite{morris2023weisfeiler, boker2023fine}.

\subsection{Almost universal invariant machine learning on point clouds via Galois theory}\label{sec:almost-universal}

Point clouds are a common data modality in several application domains, including computer vision, materials science, and cosmology. In some of these applications, each data point is a point cloud in 
$\mathbb R^{d\times n}$ modulo  translations, rotations, reflections, and permutations. Here $n$ is the number of points and $d$ is the dimension of the space.

\begin{wrapfigure}{r}{0.4\textwidth}
  \begin{center}
\vspace{-1cm}
\includegraphics[clip=true, trim={14cm 5cm 15cm 3cm}, width=0.1\textwidth]{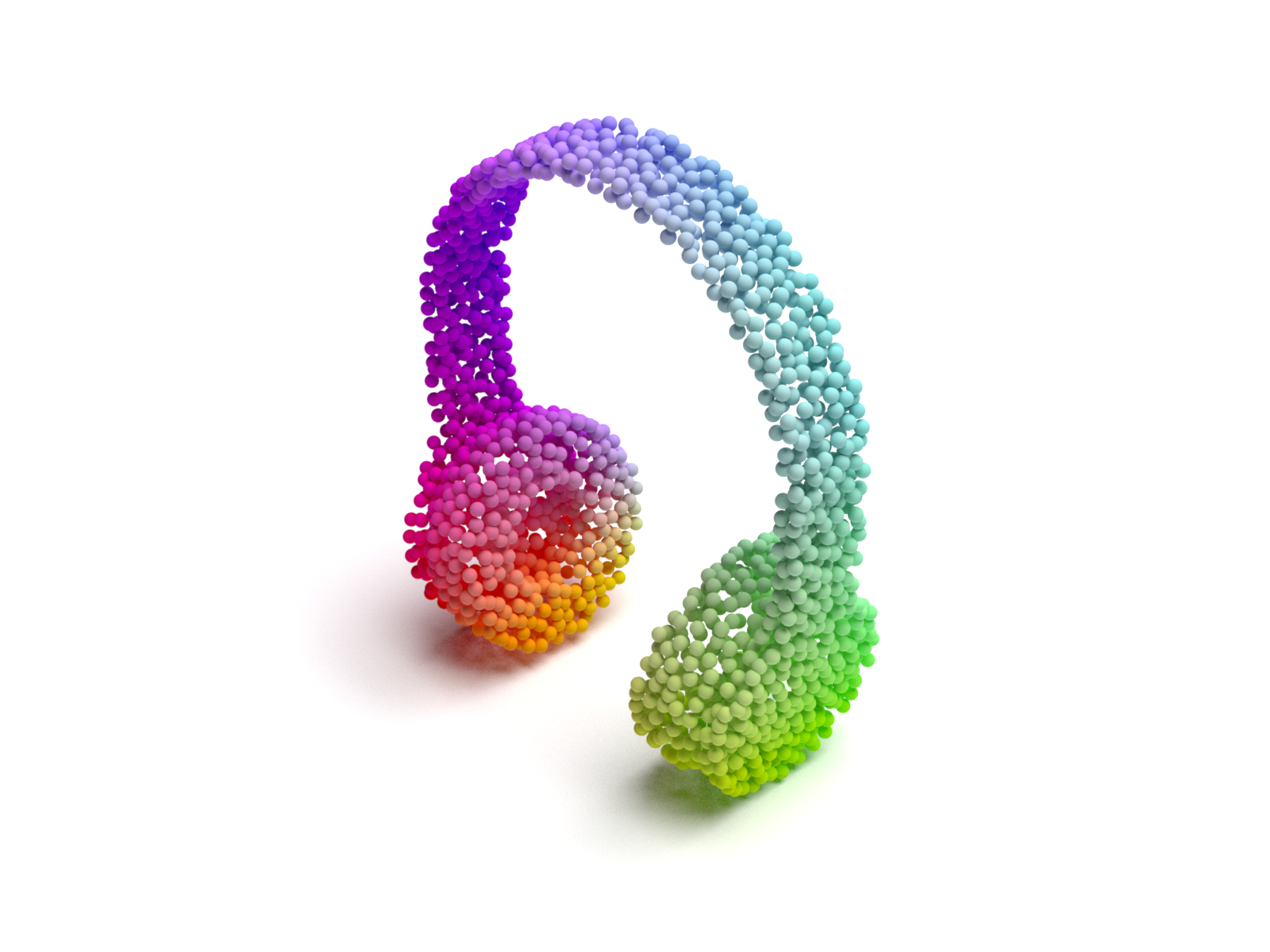}
\includegraphics[clip=true, trim={14cm 5cm 15cm 3cm},width=0.1\textwidth]{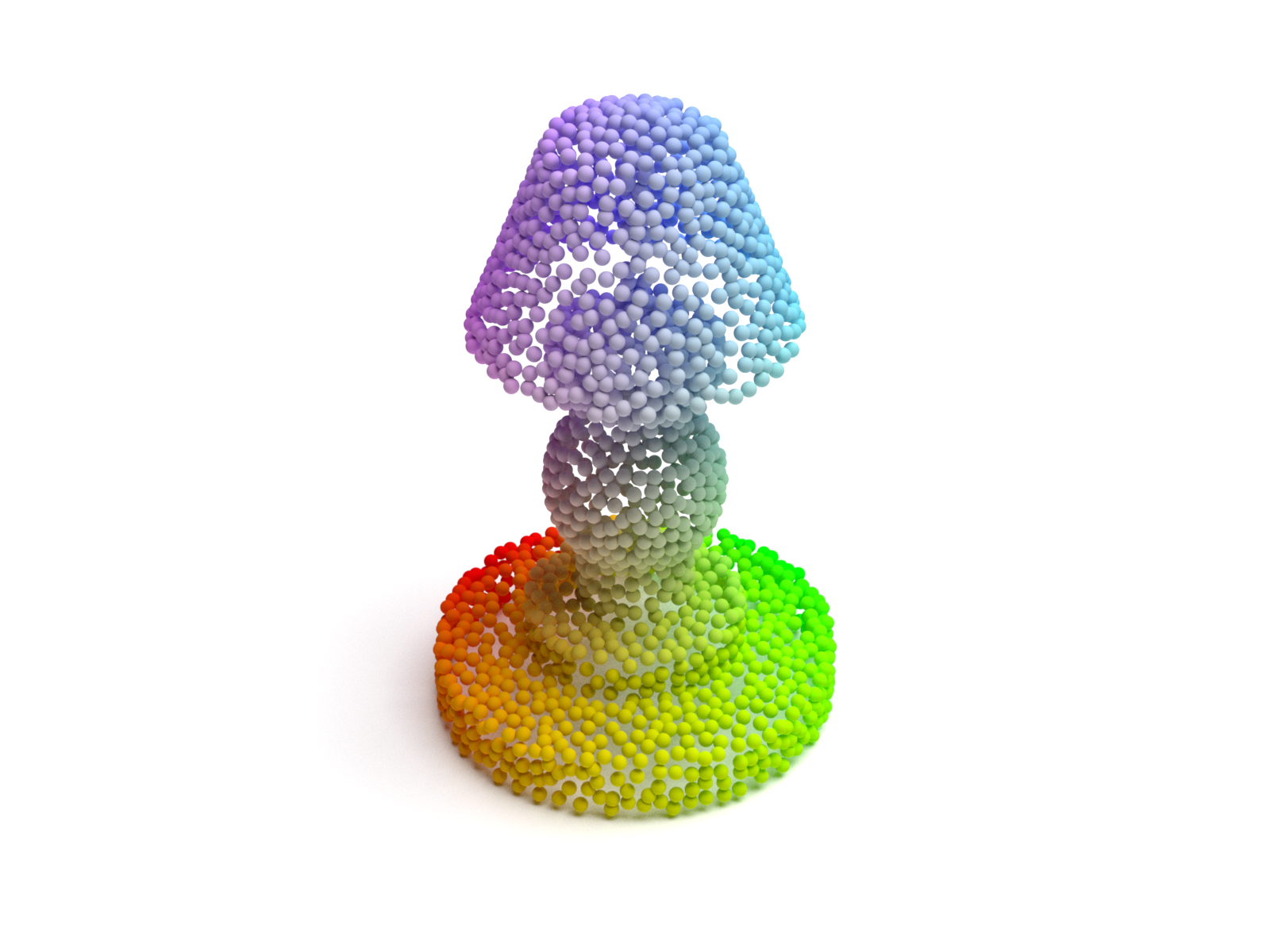}
\includegraphics[clip=true, trim={14cm 5cm 15cm 3cm},width=0.12\textwidth]{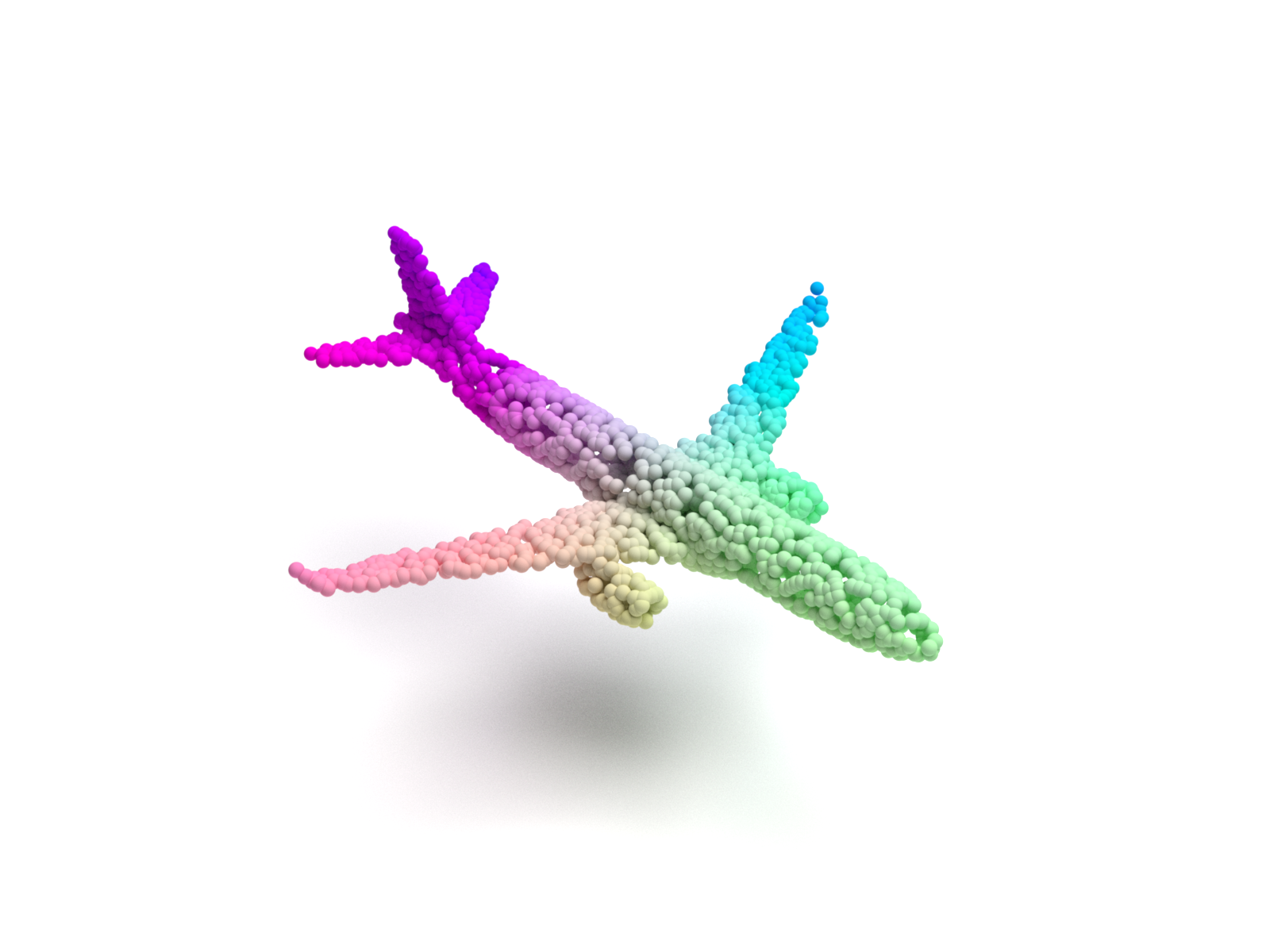}
  \end{center}
  \caption{3D objects represented as point clouds from \cite{shapenet2015}.}
\end{wrapfigure}
A function $f:\mathbb R^{d\times n} \to \mathbb R$ on a point cloud $P\in \RR^{d\times n}$ can be deformed into a function $\bar f$ invariant with respect to translations by defining $\bar f(P):= f(P-\bar P)$, where $\bar P$ is  the center of mass. The point cloud $P-\bar P$ is centered on the origin, and is the unique point cloud in $P$'s equivalence class under translations that has this property, so the new function $\bar f$ is both well-defined and translation-invariant.  This simple idea is an example of canonicalization \cite{kaba2023equivariance}, i.e., implementing invariance or equivariance via computing a canonical representative in every group orbit. This is closely related to the generalization of Cartan's notion of {\em moving frames} due to Fels and Olver \cite{fels1998moving, fels1999moving, fels2001moving, olver2003moving}.

In order to parameterize the invariant functions with respect to rotations and reflections around the origin, we can use the
Fundamental Theorem of Invariant Theory for the orthogonal group. It says that $f:\mathbb R^{d\times n} \to \mathbb R$ is $O(d)$-invariant if and only if there exists a function $h:\mathcal S(n) \to \mathbb R$ where $\mathcal S(n)$ is the space of $n\times n$ symmetric matrices with real entries satisfying  $f(P) = h(PP^\top)$. 

The function $f$ is invariant with respect to permutations of the $n$ rows of $P$ if and only if $h$ is invariant with respect to the action of permutations by conjugation on $PP^\top=:X$. To be consistent with the notation from Section \ref{sec.formal} we say that $h$ is $H$-invariant if $h(\pi X \pi^\top)=h (X)$ for any $\pi$ in the symmetric group $S_n$. (I.e., $H$ refers to the group of linear transformations of $\mathcal S(n)$ induced by the conjugation action of $S_n$; it is abstractly isomorphic to $S_n$ but the notation $H$ also specifies the action.)

The context that originally led us to the train of thought described in Section~\ref{sec:main-idea} was the work of two of the present authors in \cite{blum2024galois}. There, we implemented the $H$-invariant functions $h:\mathcal S(n) \to \mathbb R$ using the result described in Section \ref{sec.formal}. We considered a bigger group $G:=S_n\times S_{n(n-1)/2}\supset H$ which acts in $\mathcal S(n)$ permuting the diagonal and off-diagonal entries of $X$ independently. One easily constructs SW-distinguishing polynomials for the action of $G$ using symmetric polynomials. Theorem \ref{thm:galois-for-ML} allows us to extend them to a set of generically SW-distinguishing polynomials for the action of $H$ by adding a polynomial $f^*$ that Galois-distinguishes $G$ from $H$. The polynomial $f^*$ is $H$-invariant but not fixed by any $H'$ satisfying $H\subsetneq H'\subset G$.

This result provides generically SW-distinguishing invariants for the action of permutations by conjugation on symmetric matrices. Restricting the results to point clouds requires a few extra technical steps for the following reasons: (1) the symmetric matrices arising as Gram matrices of point clouds themselves form a measure zero subset of the space of symmetric matrices, and (2) the group $G$ described in the previous paragraph does not act on the point clouds nor on the subset of symmetric matrices arising from point clouds; only $H$ does. These issues are handled by using Galois theory directly. Low-rank matrix completion techniques allow us to reduce the number of invariant features to $O(d\, n)$. 

\subsection{Orbit recovery and field generation}

Another place where distinguishing and expressing have come together in data science is in a signal processing application known as {\em orbit recovery} or {\em (generalized) multi-reference alignment} \cite{abbe2017sample, perry2019sample, bandeira2020optimal, fan2021maximum, bendory2022sparse, abas2022generalized,  bendory2022dihedral, bandeira2017estimation, edidin2024orbit, edidin2024generic, bendory2025generalized}. One studies the reconstruction of a signal that has been corrupted both by noise and also by a transformation drawn randomly from a group. The corrupted signal is a random variable depending on the original signal. If the random transformation is drawn uniformly from the group, it destroys any information about where the signal lies in its group orbit, so the goal is to reconstruct the original signal's orbit, up to a small and controlled error, after witnessing many samples of the corrupted signal. A principal example is the mathematical study of {\em cryo-electron microscopy} \cite{sigworth2016principles, singer2018mathematics, bendory2020single}, a molecular imaging technique that creates many images of a molecule, each of which is both extremely noisy and also randomly oriented.

One of the findings of the literature on orbit recovery  gives  another manifestation of the {\em distinguish $\Rightarrow$ express} theme. It is shown in \cite{bandeira2020optimal, bandeira2017estimation} that, in the high-noise regime, the statistical sample complexity of the problem---in other words, the number of samples that need to be viewed for a successful approximation of the orbit to be information-theoretically possible---varies as $O(\sigma^{2d})$, where $\sigma$ is the noise level, and $d$ is the minimum degree required for the values of the group-invariant polynomials of up to that degree to uniquely identify the orbit. In other words, if (and only if) the values of the invariant polynomials of a certain degree $d$ can pin down (distinguish) the orbit of the signal information-theoretically, then this orbit can be accurately estimated (expressed) in terms of $O(\sigma^{2d})$ samples.

Because $d$ appears in the exponent, there is a strong incentive to work in regimes in which the $d$ in question can be made small. One typically gets a dramatic reduction in $d$ by working with generic rather than worst-case signals. For example, the original version of the multi-reference alignment problem is to estimate an element of $\RR^n$ to which Gaussian noise is added, and whose coordinates have also been subjected to a random cyclic shift. The implicit group action is thus the regular representation of $\ZZ/n\ZZ$. It is well-known in invariant theory that to pin down a worst-case orbit for this action requires the invariant polynomials up to degree $d=n$. On the other hand, a generic orbit (specifically, the orbit of any point whose discrete Fourier coefficients are all nonzero) can be pinned down with invariants of degree at most $d=3$. In other words, the invariants of degree $\leq 3$ are generically SW-distinguishing for $\ZZ/n\ZZ$, in the sense defined above.

It was shown by Kakarala \cite{kakarala2009completeness}, and by Smach et al \cite{smach2008generalized}, that this latter situation generalizes to the regular representation of any finite or compact Lie group (where in the latter case, regular representation means $L^2(G)$): there exists a set of invariants of degree $3$, known as the {\em bispectrum}, that uniquely identifies a generic orbit. Thus,  {\em distinguish $\Rightarrow$ express} happens in degree 3, where by {\em distinguish} we mean generically SW-distinguish, and {\em express} is in the sense of the orbit recovery problem: the orbit can be well-approximated using $O(\sigma^{2\cdot 3})$ samples.

Because of Rosenlicht's theorem, which asserts that generic SW-distinguishing is related to expressing in the sense of generating a field---these results led to the question of whether the polynomials of degree $\leq 3$ actually generate the field of rational invariants, at least if $G$ is finite (so that the regular representation is finite-dimensional). The bispectrum consists of functions that are polynomial over $\RR$ but not over $\CC$; since $\RR$ is not algebraically closed, Rosenlicht's theorem does not immediately imply a field generation result. So there was a real question.

The answer has turned out to be yes. For $G$ abelian, this was shown in \cite{bandeira2017estimation} by way of Galois theory: the invariants of degree $\leq 3$ were shown to be Galois-distinguishing for $G$, and field generation follows by the Fundamental Theorem, as discussed above. But this was dramatically generalized in \cite{edidin2024generic}. The technique was wholly different, but equally thematic from the present point of view. If $G$ is any finite group, the ground field is any infinite field (for example, $\QQ$, $\RR$, or any algebraically closed field), and the space of signals on which $G$ acts is, or even just {\em contains}, the regular representation, then \cite{edidin2024generic} showed that the polynomial invariants of degree at most $3$ are generically distinguishing for $G$.\footnote{Here {\em generically distinguishing} means that they distinguish all the orbits in a nonempty Zariski-open subset of signal space; when the ground field is $\RR$ or $\CC$, this is equivalent to them being generically SW-distinguishing in the sense defined above.} In characteristic zero (and in particular, over $\CC$), field generation follows by Rosenlicht's theorem.

\subsection{The number of (almost) distinguishing invariants}

In invariant theory, the main object of interest is the ring $k[\mathcal V]^G$ of invariant regular functions of a group $G$ acting on an algebraic variety $\mathcal V$, and the first step in taking a hold of it is to find a set of algebra generators. As discussed above, more relevant to data science applications (via the Stone--Weierstrass Theorem) is a set of orbit separators, and this is good because separating sets can be much smaller and easier to compute---see the remark at the end of Section~\ref{sec:main-idea}, and the below. If one is willing to jettison a small ``bad'' subset of $\mathcal V$, then smaller and easier to compute still may be a set of {\em generic} orbit separators. Here we discuss methods to compute small sets of separators and generic separators.

If the group $G$ is compact, the algebra generators are also orbit separators. While the number of algebra generators required may be large, in general no more than $2D+1$ orbit separators are needed, where $D$ is the dimension of the orbit space $\mathcal V/G$ \cite{dufresne2009separating}. If the group $G$ is finite, $D$ is just the dimension of $\mathcal V$, and if $G$ has positive dimension  then in general $D$ is lower still. However, this theorem is proven by starting from an arbitrary set of orbit separators (such as generators) and linearly combining them, using dimension-counting to show that a small number of such linear combinations remains separating. Thus, it does not provide a method to actually compute $2D + 1$ separators without first computing a larger separating set.

One method for tackling this challenge was provided in recent work of Dym and Gortler \cite{dym2025low}, which uses techniques developed for phase-retrieval \cite{balan2006signal}. 
They show that one can efficiently compute $2D+1$ separators by sampling randomly from a parametrized family of invariants fulfilling a certain criterion they call {\em strong separation}: this means that for any two elements of $\mathcal V$ in different orbits, the invariants of the parametrized family that fail to distinguish them are parametrized by a small (specifically, of positive codimension) subset in the parameter space. For some examples, it is possible to efficiently construct such strongly separating families without having prior access to a finite separating set. For example, Dym and Gortler show \cite[Proposition~2.1]{dym2025low} that for the action of the symmetric group $S_n$ on $n\times d$ matrices, the family of eminently computable invariants $X\mapsto \langle u, \operatorname{sort} (Xv) \rangle$, parametrized by $(u,v)\in \RR^n\times \RR^d$, is strongly separating.\footnote{A related approach is taken in recent work of Cahill, Iverson, Mixon, and Packer \cite{cahill2025group}, which also reduces the number of separators by one, if the group is finite and $\mathcal V = \RR^D$: this paper shows that, for any finite group, a set of $2D$ invariants selected randomly from a certain parametrized family of piecewise-linear maps on $\mathcal V$ is separating. The approach is further studied in \cite{mixon2023max, balan2023g, mixon2025injectivity}. These invariants have the added advantage of being numerically stable (in particular, the induced map on $\mathcal V/G$ is bilipschitz). They are not a priori efficient to compute if the group $G$ is large, but can be efficiently computed in certain special cases.}

Dym and Gortler's method can also be used to extract $D+1$ {\em generic} separators from a (generically) strongly separating family of functions as well. The $D+1$ bound is optimal in general: usually, $D+1$ invariants are required for generic separation, although in special cases $D$ may suffice.

Another approach to finding $D+1$ generic separators is via Rosenlicht's Theorem,  which identifies the problem with finding a generating set for the invariant field $k(\mathcal V)^G$. From this point of view, the bound $D+1$ can be viewed as a consequence of the {\em primitive element theorem} from field theory. The latter asserts that a finite, separable field extension can always be generated by a single element. It follows that a generating set for the field $k(\mathcal V)^G$ of rational invariants requires, at worst, the $D$ elements of a transcendence basis for $k(\mathcal V)^G$ over $k$, plus one additional element.

General methods for computing field generators for the rational invariants of actions of algebraic groups on varieties are given in \cite{muller1999calculating, hubert-kogan, kemper2007computation}, based on Gr\"obner basis methods. They do not achieve the optimal number $D+1$ of generators, and (due to the Gr\"obner bases) are not computationally efficient, but they are fully algorithmic and very flexible with respect to the group and the action. The method of Hubert and Kogan \cite{hubert-kogan} also provides an algorithm to express other invariants in terms of the generators. There are more efficient methods for finding field generators for specific types of group actions, such as \cite{hubert2012rational, hubert-labahn, gorlach2019rational, hubert2025algebraically}, and explicit generators are known in various special cases, some going back to the origins of Galois theory.

Through the lens of the Stone--Weierstrass theorem, the sets of algebra generators and the smaller separating sets of $2D+1$ elements have the same expressive power. However, this is a coarse claim that does not consider the (potentially different) approximation rates, which are generally not known. Even at this coarse level, though, sets of generic separators (e.g., field generators or other SW-distinguishing sets) have weaker expressive power. Not all generic separators have the same expressive power. For example, different generic separators may fail at different closed, zero-measure, $G$-stable sets. 

\section{{\em Distinguish $\not \Rightarrow$ express}: non-examples of the theme}\label{sec:non-ex}

Not every statement of the form {\em distinguish $\Rightarrow$ express} that one might hope for actually holds: everything depends on how the general notions of {\em distinguish} and {\em express} are pinned down and operationalized in a given context.

\subsection{Invariant theory} \label{sec:counter-math}
     
    In invariant theory, an important counterexample to the general principle is as follows. If $G$ is a finite group and $V$ a finite-dimensional vector space over a field $\kk$, then generators for the algebra $\kk[V]^G$ of invariant polynomials on $V$ do in fact separate the orbits of $G$ on $V$---i.e., if you can express in the sense of algebra generation, then you can distinguish. However, a set of invariant polynomials able to separate orbits does not necessarily generate $\kk[V]^G$ as an algebra. So if we ask too much of the word \emph{express} in this context, then \emph{distinguish $\Rightarrow$ express} may fail.
    
    In the more general situation that $G$ is a linear algebraic group, then generators for $\kk[V]^G$ may not necessarily distinguish orbits of $G$ on $V$: orbits may fail to be closed, and a non-closed orbit cannot be distinguished by any invariant polynomial from any other orbit whose closure intersects its closure. So if we ask for too much from the word {\em distinguish}, then even the converse theme, {\em express $\Rightarrow$ distinguish}, may fail. 
    
    There is a standard relaxation that rescues {\em express $\Rightarrow$ distinguish} in this context: there is an algebraic variety $V\sslash G$, the {\em categorical quotient} or {\em GIT quotient}, that is universal (in the category of algebraic varieties) with respect to receiving a morphism from $V$ constant along orbits of $G$; and generators for $\kk[V]^G$ do separate the points of $V\sslash G$. (In the case of $G$ finite, $V\sslash G$ is precisely the orbit space.) However, even under this relaxation, {\em distinguish $\Rightarrow$ express} does not hold: invariant functions that separate the points of $V\sslash G$ (known as a {\em separating set}) may fail to generate $\kk[V]^G$.
    
    A simple example---even for a finite group over an algebraically closed field of characteristic zero---is given by $\kk=\CC$ and $G=\ZZ/n\ZZ$, acting faithfully on $V = \CC^2$ by scalar matrices. In this case the orbit of a point consists of its images under scalar multiplication by an $n$th root of unity, and the invariant algebra $\CC[x,y]^G$ is the $n$th {\em Veronese ring}, the subalgebra of $\CC[x,y]$ spanned by monomials with total degree a multiple of $n$. It is generated (as an algebra) by the monomials of degree exactly $n$, of which there are $n+1$. This is a minimal generating set: one can see by considerations of degree that none of them can be expressed as a polynomial in the others. However, the three monomials $x^n$, $x^{n-1}y$, and $y^n$ already separate orbits. The values of $x^n$ and $y^n$ already pin down $x,y$ up to multiplication by (possibly unrelated) $n$th roots of unity, and then the value of $x^{n-1}y$ pins down the relation between the two. (Indeed, $x^{n-j}y^j$ would do the same for any $1\leq j \leq n-1$ relatively prime to $n$.) This example shows that a separating set can be much smaller than a generating set. It is also possible for it to be much lower degree. While generating sets are the traditional object of study of invariant theory, separating sets became their own locus of interest around 20 years ago, and there is now a significant literature on them, e.g., \cite{ derksen-kemper, domokos2007typical, dufresne2008separating, draisma2008polarization, kemper2009separating, sezer2009constructing, dufresne2009separating, dufresne2009cohen, kohls-kraft, elmer2012separating, dufresne2013finite, kohls2013separating,  dufresne2015separating, domokos, reimers2018separating, reimers2020separating, lopatin2021separating, domokos2022separating, kemper2022separating, domokos2024separating, schefler2025separating}.
    
\subsection{Understanding expressivity limits for equivariant architectures}\label{sec:shubhendu}

If we consider machine learning models $\mathcal{X}\to \mathcal Y$, where $\mathcal X$ is equipped with an action by a group $G$, and we know an invariant ``feature extraction'' map $\Phi:\mathcal X\rightarrow \RR^{n_0}$ that fully separates the orbits, then we can obtain a universally approximating invariant architecture. We can do this by applying a universal architecture, such as an MLP with a single hidden layer of unbounded width, to the output of the feature extraction map (i.e., $M_2\circ\sigma\circ M_1\circ \Phi$ in the notation of Section \ref{subseq.sw.ml}).

However, machine learning architectures are often implemented in a different way. In practice, invariant and equivariant machine learning architectures are commonly implemented as compositions of equivariant layers. In the case of a feedforward network, this looks as follows: the group (necessarily finite in this situation; sometimes a discretization of an infinite group) acts through permutations on $\mathcal X$, $\mathcal Y$, and all the hidden layers as well, and the individual $M_j$'s are constrained so that the maps $M_j:L_{j-1}\rightarrow L_j$ in \eqref{eq.mlp} are equivariant. (The pointwise nonlinearity $\sigma$ is also equivariant because applying the same function to every entry in a vector commutes with permutations of the entries.) In the invariant case, the final step $M_{s+1}$ is an invariant pooling operator that makes final output invariant. Also, neither equivariant feedforward networks nor other architectures are actually universal in practical implementations, since they are defined using a fixed finite set of parameters. As stated at the beginning of Section~\ref{sec:non-ex}, the general theme {\em distinguish $\Rightarrow$ express} is constrained by the way that these notions are operationalized through a given machine learning architecture; specific implementation choices may lead to a situation where distinguishing is insufficient for expressing. This topic has recently been explored in \cite{pacini2025universality,pacini2025universality2} for architectures based on equivariant layers. 

The work \cite{pacini2025universality} studies the question of how design choices such as width and depth of an equivariant feedforward network influence the ability of the network to translate separation power into expressivity. It gives examples of invariant models with identical distinguishing power but different approximation capabilities. In particular, the authors compare three architectures, all of which have $\RR^n$ as input space, a single hidden layer of unbounded width, a single number as output, and are arranged to parametrize only functions that are invariant with respect to the group $S_n$ of permutations of the coordinates of the input---but they differ in the way that this group acts on the hidden layer and/or in the constraints on $M_1$:
\begin{enumerate}
    \item The hidden layer is some number $h$ of copies of $\RR^n$, each receiving the canonical permutation representation of $S_n$. The map $M_1$ has components $M_{1,1},\dots,M_{1,h}$; each $M_{1,i}$, $i=1,\dots,h$, applies the same fixed affine map $x\mapsto a_ix+b_i$ to every entry of the input vector and sends it to the corresponding entry of the $i$th copy of $\RR^n$ in the hidden layer.\label{arch:1-conv}
    \item The hidden layer is again $h$ copies of $\RR^n$ each receiving the canonical permutation representation. This time, $M_1$ is allowed to be any equivariant affine map from the input layer to the hidden layer.\label{arch:pointnet}
    \item The hidden layer is $h$ copies of the regular representation of $S_n$, and $M_1$ is any equivariant affine map from the input layer to the hidden layer.\label{arch:regular}
\end{enumerate}
In all three cases, the map $M_2$ is allowed to be any invariant affine function on the hidden layer, and the parameter $h$ is allowed to be unboundedly large. The main finding of \cite{pacini2025universality} is that, while all three of these architectures are able to fully distinguish the orbits of $S_n$ on $\RR^n$ \cite[Proposition~9]{pacini2025universality}, the families of functions uniformly approximated by them have strict containments \ref{arch:1-conv}$\subsetneq$\ref{arch:pointnet}$\subsetneq$\ref{arch:regular} if $n\geq 3$ \cite[Proposition~16]{pacini2025universality}. Only the third architecture can approximate every $S_n$-invariant continuous function, i.e. it is the only one to witness the {\em distinguish $\Rightarrow$ express} principle; the authors say it is {\em separation-constrained universal}, while the first two are not. The method of proof of the strict containments involves a sufficient condition for expressivity failure in terms of differential operators, based on the theory of ridge functions.

The fact that architecture~\ref{arch:regular} does witness {\em distinguish $\Rightarrow$ express} (i.e., is separation-constrained universal, in the authors' language) is a manifestation of a general principle that a model with a single hidden layer where the group acts as sufficiently many copies of the regular representation, is universal \cite{ravanbakhsh2020universal}. This follows from the classical universality theorem mentioned above, because a hidden layer consisting of copies of the regular representation can implement taking the average over the group of a function modeled by a single-hidden-layer MLP; the classical result says the latter architecture can $\varepsilon$-approximate a desired invariant function, and then averaging enforces that the approximation is also invariant.

Loosely motivated by the equivariant setting, in which separation power {\em per se} is an inadequate framework to reason about approximation ability---indeed, if $\mathcal X =\mathcal Y$, any parametrized family of equivariant maps that includes the identity map automatically separates all orbits---the subsequent work \cite{pacini2025universality2} explores a different notion of distinguishing power and its relation to expressivity. Given a class of functions $\mathcal F$ on an input space $\mathcal X$, the {\em equivalence relation $\rho$ on $\mathcal X$ induced by $\mathcal F$} is defined by $x_1\sim_\rho x_2$ if $f(x_1)=f(x_2)$ for all $f\in \mathcal F$; i.e., $\rho$ captures the distinguishing power of the class $\mathcal F$. The authors work with a refinement of this notion: given a class of maps $\mathcal F$ from $\mathcal X $ to $\mathcal Y:=\RR^n$, they consider the tuple $\overline \rho =(\rho_1,\dots,\rho_n)$ of equivalence relations on $\mathcal X$ induced by each of the function classes $\{\pi_i\circ f:f\in \mathcal F\}$, where $\pi_1,\dots,\pi_n$ are a set of projections from $\mathcal{Y}$ to each of its coordinate functions. Then, given such a tuple $\overline \rho$ of equivalence relations, they consider the set $C_{\overline \rho}(\mathcal X,\mathcal Y)$ of all continuous maps $f$ such that $\pi_i\circ f$ respects $\rho_i$ for each $i=1,\dots,n$.  They say an architecture is {\em entrywise separation-constrained universal} if the function class $\mathcal F$ it parametrizes can approximate (uniformly on compact subsets of $\mathcal{X}$) any function in $\mathcal C_{\overline \rho}(\mathcal X,\mathcal Y)$. They then prove some results \cite[Theorems~2 and~3]{pacini2025universality2} that show that adding certain hidden layers can force architectures to be entrywise separation-constrained universal. This is another form of the {\em distinguish $\Rightarrow$ express} principle.

Regardless of whether the machine learning models are based on invariant features or equivariant layers, the  {\em distinguish $\Rightarrow$ express} condition requires that complexity (and therefore the number of parameters needed to implement such functions) is unbounded. This is not a realistic assumption and the expressivity limits under realistic computational constraints is an interesting research direction. Recent work \cite{siegel2026quantitative} provides quantitative approximation results, bounding the needed number of parameters the model requires to achieve a given error threshold, for certain equivariant architectures.

\section{{\em Distinguish $\Rightarrow$ express} in linguistics}\label{sec:linguistics}

The version of the {\em distinguish $\Rightarrow$ express} principle that applies to natural (human) language reflects a foundational tenet, articulated by Saussure \cite{saussure2006writings}, that the basic unit of language is the \emph{linguistic sign}, defined as the unit that pairs a string of sounds\footnote{ More precisely, Saussure spoke of the {\em sound-image}, the psychological abstraction over individual pronunciations. He only discussed spoken languages, but the same principle applies to signed and written languages as well.} with a concept (or cluster of meaning components). We below use the term \emph{encode} to refer to the relationship between the string of sounds and the concept. As discussed here, the sign corresponds roughly to a word, and, according to one interpretation of Saussure's theory, has no fixed meaning, but has meaning only in contrast to that of other signs. In this strong version of the claim, a word has no inherent capacity to express except within its system of distinguishing.\footnote{Our primary method is to describe expressed distinctions in analogous systems from different languages and reveal the corresponding gaps in those languages. However, the same point can be made within a language: consider the meaning range of English {\em red} as it applies to wood, hair, and cabbage: its meaning (expressive range) exists in contrast to other possible hues within the respective domains.}

We do not consider the converse principle {\em  express $\Rightarrow$ distinguish} here, because it would have forced us to wade into considerations beyond our scope about the range of meaning of the word {\em express} in the context of linguistics.

We begin from the premise that everything expressible in language is expressible in all languages.\footnote{As this paragraph will make clear, this premise is an assertion about whole languages, rather than the individual systems of distinction within a language that interact to result in the expressive power of the language as a whole.} Languages differ in their systems of distinction, not in their capacity to express. Given a particular system of meaning contrasts, some languages provide more or fewer distinctions within that system than the corresponding system in other languages (and in general, fewer than is logically possible). 
This holds across all levels of language:  sounds (phonetics/phonology), word structures (morphology), word combinatorics (syntax); meaning (semantics/pragmatics). But in language, individual systems of distinction do not function in isolation from one another. If the expression of a certain concept depends on a certain distinction within some system of meaning contrasts in one language, but that distinction is unavailable inside the analogous system of meaning contrasts in another language, other systems in the latter language can be recruited for the purpose. We will see examples of this below. Our focus is on economies of definable systems of distinction, not the expressive limits of a language as a whole. 

In view of this, in order to demonstrate the interplay of distinction and expression, we will be looking at conventionalized systems of distinction within the constrained domain of word meaning. Our examples include: kin terms, which express biological and social aspects of relationships within a social/family unit (e.g., mother; father); personal pronouns; cardinal number systems; and color terms. Working within this narrow scope allows us to witness how the ability to express a concept emerges from a system of distinctions. Comparison of the systems of distinction drawn by kin terms in different languages reveals contrasts in what they can express.  Some of the examples also illustrate how expressive limitations in one system of distinctions can be overcome by recruiting other systems, as discussed above. The words play the role of functions in the Stone--Weierstrass theorem and field elements in the Galois theory context: they are the primitives out of which expression is built, and whose distinguishing power is harnessed to actualize that expression.\footnote{An additional feature of this analogy is that in order to express, the primitives (words) are combined according to a constrained set of combination rules (the language's syntax and morphology), just as the primitives in the Stone--Weierstrass theorem were combined according to the fixed set of combination rules $\{\RR\text{-linear combination},\times,\varepsilon\text{-approximation}\}$, and in the Galois context according to the fixed set $\{\RR\text{-linear combination},\times,\div\}$. However, we will not focus on linguistic combination rules in this discussion.}

We are interested for the sake of this discussion in particular instances of the Saussurean sign, namely, \emph{ irreducible words}. An irreducible word is a lexical sign whose meaning components cannot be aligned with subunits of the sound part of the sign, i.e., we do not include ``great-grandmother''-type words in our purview. Irreducibility is a property of a word as it is used contemporaneously, not a comment on its etymology: {\em y'all} functions as an irreducible word in English, despite having recognizable etymological components, as demonstrated by the fact that {\em y'all} can be used in contexts where the quantifier {\em all} cannot--such as when the word refers to a single person or two people.

While we hope the previous sections have convinced the reader that the interplay between distinction and expression is a theme connecting disparate areas in mathematics, in classical linguistic theorizing it is foundational. We give a few examples of how the members of a definable system within a language place contrastive pressure on one another and how expressivity emerges from that pressure within the system.

In the following examples we have simplified the descriptions to make unfamiliar systems of meaning distinction tractable for non-linguists.

\begin{example}[Kin terms]
 Kin term systems consist of signs that encode combinations of meaning components that express socially significant relationships between people. Kin terms may combine the individual property\footnote{The linguistic term would be {\em inherent property}, but this term does not enter into any debate about sex vs. gender or about the inherency of sex. The word {\em inherent} in this context just communicates that the sex/gender feature holds of the individual, regardless of any kin relationships. (The contrasting term is {\em relative}, not {\em extrinsic}. Taking a step back, this example is another illustration of Saussure's contention that meaning depends on the system of distinctions.) Throughout, we use ``\textsc{sex}'' to refer to this meaning component.} of \textsc{sex}\footnote{Here and below, we use the convention that components of meaning are rendered in small capital letters.} with properties relative to the reference person (=\textsc{ego}). The relational components of most kin term systems are:  \textsc{age}; \textsc{lineality} (older or younger generation); \textsc{collaterality} (sibling); \textsc{marital relation} (spouse/in-law); and \textsc{residence}. 
No such system includes more than a fraction of the logically possible distinctions \cite{nerlove1967sibling}. English, for example, has words for siblings that encode a distinguisher of sex, i.e., {\em brother} and {\em sister}, but not for the age of that sibling relative to  \textsc{ego}. However, in Javanese, siblings are differentiated by age relative to \textsc{ego}, and, among older siblings, also for sex; see Table~\ref{table:sibling-english-javanese}.  Whereas the English system fails to distinguish the sibling’s age relative to \textsc{ego} while distinguishing on sex, the Javanese system does the converse in the case of younger siblings. 

\begin{table}[h]
\centering
\renewcommand{\arraystretch}{1.5} 
\begin{tabular}{|c|c|c|c|c|c|}
    \hline
    \multicolumn{2}{|c|}{English} &  & \multicolumn{2}{c|}{Javanese} &  \\ \cline{1-2} \cline{4-5}
    \textsc{Male} & \textsc{Female} & & \textsc{Male} & \textsc{Female} & \\ \cline{1-2} \cline{4-6}
    \multirow{2}{*}{brother} & \multirow{2}{*}{sister} & & (kang)mas & mbak(yu) & \textsc{older than ego} \\ \cline{4-6}
    & & & \multicolumn{2}{c|}{adhik} & \textsc{younger than ego} \\ \hline
\end{tabular}
\caption{Sibling terms in English \& Javanese}
\label{table:sibling-english-javanese}
\end{table}

Indonesian uses a yet different system, including a distinct word, {\em besan}, for the relationship between a parent of one spouse and a parent of the other spouse. This word is more specific than the English cover term {\em in-law}, which can be used in reference to this relative, but which is also used for a much wider set of relatives without differentiation. That is, a sister-, cousin-, or other in-law is covered by this general term. In Indonesian, there is no term that covers all in-laws. These two contrasting systems show that any indirect marital relation is expressed by a single, general term in English, while no term of Indonesian is correspondingly underspecified. Indonesian makes distinctions that English ignores, and vice versa, within their kin term systems.

Again, we are interested here in the words that are irreducible, i.e., not ``great-grandmother''-type words. The set of Indonesian lineal kin terms provides another example, with irreducible words for relatives related generationally. This system, shown in Table~\ref{table:kin-indonesian}, contrasts with that of English, which includes words that encode both generationality and sex, but requires composition (i.e., the formation of reducible words) to reach more than a generation of removal from the ego in either direction. In Indonesian, only the two lineal positions immediately above the ego encode sex distinction, while non-compositional words encode as many as four generations ascending and descending, without sex distinction.\footnote{Note, in contrast to the English system, the relevant distinction in the Indonesian system may allow underspecification of the direction of removal from \textsc{ego}, encoding simply 3rd ({\em buyut}) or 4th ({\em canggah}) generation of removal.}

\begin{table}[h]
\centering
\renewcommand{\arraystretch}{1.5} 
\begin{tabular}{|l|c|c|}
    \hline
    \multicolumn{1}{|c|}{} & \textsc{Male} & \textsc{Female} \\ \hline
    \textsc{4 generations ascending} [great-great-grandparent] & \multicolumn{2}{c|}{canggah} \\ \hline
    \textsc{3 generations ascending} [great-grandparent] & \multicolumn{2}{c|}{buyut} \\ \hline
    \textsc{2 generations ascending} [grandparent] & kakek & nenek \\ \hline
    \textsc{1 generation ascending} [parent] & bapak & ibu \\ \hline
    \textsc{ego} & \multicolumn{2}{c|}{} \\ \hline
    \textsc{1 generation descending} [child] & \multicolumn{2}{c|}{anak} \\ \hline
    \textsc{2 generations descending} [grandchild] & \multicolumn{2}{c|}{cucu} \\ \hline
    \textsc{3 generations descending} [great-grandchild] & \multicolumn{2}{c|}{cicit/(buyut)} \\ \hline
    \textsc{4 generations descending} [great-great-grandchild] & \multicolumn{2}{c|}{piut/(canggah)} \\ \hline
\end{tabular}
    \caption{Indonesian lineal kin terms}
    \label{table:kin-indonesian}
\end{table}

Thus, even within the domain of kin terms, we have three very different examples of the principle {\em distinguish $\Rightarrow$ express}. 
\end{example}

\begin{example}[Personal pronouns]
The next example is from pronoun systems. English has words for speaker/first person, hearer/2nd person, and other, with distinct forms for subject, object, and possessor. English distinguishes among masculine, feminine, and neuter only in the third person singular. Many personal pronoun systems in other languages provide more or different distinctions than English does. Tok Pisin, an English-based creole\footnote{A {\em creole} is a natural language that results from a pidgin language that speakers acquire natively. A {\em pidgin} is a simplified language variety that results from prolonged interaction between two or more language communities.} of Papua New Guinea, is one such language. Its first person plural pronouns include the component of {\em clusivity}: whether the hearer is included in or excluded from the pronoun’s reference. English lacks this distinction, which may lead to uncertainty, as shown in this example, which contains a compensation for this vagueness: \\

\noindent A: We’re supposed to finish up by noon.\\

\noindent B: You mean you and me or you and her?\\

\noindent A: You and me! \\

Such underspecification is impossible in Tok Pisin, whose two words for \textsc{first person plural} encode this expressive distinction.

In addition, rather than a two-way distinction between singular and plural as in English, Tok Pisin's pronoun system includes dual and trial forms, making a four-way distinction in grammatical number.\footnote{Verhaar \cite{verhaar1995toward} notes that both quadral ({\em yufopela}) and quintal ({\em yufaipela}) forms are attested, but rare. This has implications for the current discussion in that speakers may be able to productively apply a linguistic formula to increase the distinguishing power and thereby expressiveness.} So Tok Pisin both adds distinctions to the number dimension and adds the inclusive/exclusive distinction to the first person pronouns. We show these forms in Table~\ref{table:tok-pisin}.

\begin{table}[h]
\renewcommand{\arraystretch}{1.5} 
\centering
\begin{tabular}{|l|c|c|c|c|}
    \hline
    & \textsc{Singular} & \textsc{Dual} & \textsc{Trial} & \textsc{Plural} \\ \hline
    \textsc{First} & mi & & & \\ \hline
    \textsc{First Inclusive (1st \& 2nd)} & \multirow{2}{*}{} & yumitupela & yumitripela & yumi \\ \cline{1-1} \cline{3-5}
    \textsc{First Exclusive (1st \& not-2nd)} &  & mitupela & mitripela & mipela \\ \hline
    \textsc{Second} & yu & yutupela & yutripela & yupela \\ \hline
    \textsc{Third} & em & tupela & tripela & ol \\ \hline
\end{tabular}
    \caption{Tok Pisin personal pronouns \cite{verhaar1995toward}}
    \label{table:tok-pisin}
\end{table}

    As an English-based language having undergone creolization, Tok Pisin’s pronouns are, in their current use, irreducible, though historically they result from combining elements (just like the English word {\em y'all}, discussed above).  The components are derived from the English words {\em me} ({\em mi}), {\em you} ({\em yu}), {\em two} ({\em tu}), {\em three} ({\em tri}), {\em fellow} ({\em pela}), {\em him} ({\em em}), and {\em all} ({\em ol}). Relative to their English origins, some of the meaning components of the Tok Pisin words have undergone linguistic {\em abstraction}---that is, they have lost elements of their earlier meaning. For example, Tok Pisin {\em mi}  means \textsc{first person singular} whereas English {\em me} means \textsc{first person singular object} (as distinct from {\em I} [\textsc{subject}] or {\em my} [\textsc{possessive}]). By contrast, Tok Pisin {\em yumi} consists of two elements each of which has lost a distinguisher of the English etymon, and added the distinguisher of clusivity (a function borrowed from the indigenous contact languages of New Guinea) to create a single, more specific, unit, meaning ‘you, me, and others.’ It is specifically not a translation of {\em you and me}, which is correctly rendered by {\em yumitupela}.  The third person singular Tok Pisin pronoun {\em em} has abstracted away from \textsc{sex}, thereby losing a distinction of the English system.
\end{example}

\begin{example}[Number systems]
Some languages make only very few number-counting distinctions; as few as the distinction between \textsc{one} and \textsc{more than one}.\footnote{This is the same system of distinction found in English {\em grammatical} number, as in \textsc{singular} vs. \textsc{plural}. Some languages have richer systems of grammatical number.} Maybrat, a Papuan language of Indonesia, has a base-5 number system that expresses only the numbers one through five. When expressing higher values is required, a small subset of body part terms (e.g. ‘finger’) is redeployed (with attendant abstraction) as components of the counting system in combination with the number words. Combining elements of these two systems in a rule-governed way results in the conventional counting system shown in part in Table~\ref{table:maybrat}, allowing for the expression of additional numbers. This example illustrates the principle  discussed above, that systems of distinction in language do not work in isolation:  a lack of distinguishing power among the primitives in a given system of distinctions may be overcome by recruiting other systems of distinction and using the language's combination rules.\footnote{Traditionally, this counting system stopped around 80, with a switch to another language now used for higher numbers.}

\begin{table}[h]
    \centering
    \renewcommand{\arraystretch}{1.5}
    \begin{tabular}{|c|l|l|}
    \hline
    Numerical Gloss & Maybrat & English Gloss \\ \hline
    1  & sau         & one               \\ \hline
    2  & ewok        & two               \\ \hline
    3  & tuf         & three             \\ \hline
    4  & tiet        & four              \\ \hline
    5  & mat         & five              \\ \hline
    6  & krem sau    & finger one        \\ \hline
    7  & krem ewok   & finger two        \\ \hline
    8  & krem tuf    & finger three      \\ \hline
    9  & krem tiet   & finger four       \\ \hline
    10 & st-atem     & my-hand           \\ \hline
    11 & oo krem sau & foot finger one   \\ \hline
    12 & oo krem ewok & foot finger two  \\ \hline
\end{tabular}
    \caption{Maybrat number system \cite[pp.~108--110]{dol2007grammar}}
    \label{table:maybrat}
\end{table}
\end{example}

\begin{example}[Color term systems]
The English color term system comprises a set of words that empirically approaches the maximum number of distinctions attested among the world’s languages \cite{berlin1991basic, wals-133}; it happens that the distinctions are primarily based on hue. Other languages’ systems make fewer distinctions, and may make those distinctions based on other elements of color such as shade and saturation. The smallest number of members of a color term system is two, exemplified by Dagum Dani, a language of West Papua, Indonesia \cite{heider1972structure}.\footnote{Dagum Dani distinguishes light + warm (yellow/red) from dark + cool (blue/green) colors.} The example we present here is from Nafaanra, a member of the Niger-Congo language family spoken in Ghana and Ivory Coast. We have the unique opportunity to describe it in terms of two phases, separated by 40 years, which shows an evolution from a three-term system to a ten-term system; see Figure~\ref{fig:nafaanra-color}. The Nafaanra 1978 system distinguishes terms for \textsc{light}, \textsc{dark}, and \textsc{warm/red-like}, which does not take hue as a primary distinguisher \cite{zaslavsky2022evolution}. The Nafaanra 2018 system distinguishes seven categories in addition to the original three, which have accordingly shrunk in their expressive range---thus arriving at a similar number of irreducible color terms as in the English system, although the expressive ranges of the individual terms differ from those in English. As shown in the figure, a meaning component, \textsc{hue}, which was not primarily significant in the 1978 system, has become a primary distinguisher in the 2018 system. In essence, Nafaanra has added a dimension of distinction that was not part of the original expressive system. This additional dimension of distinction is likely the result of contact between Nafaanra and Twi and English, both of which have hue-based color term systems that are more differentiated than Nafaanra had in 1978.\footnote {As \cite{zaslavsky2022evolution} note, the Nafaanra system is not a result of borrowing of either the terms or their meaning ranges wholesale from either English or Twi.} 

\begin{figure}[h]
\begin{center}
\includegraphics[width=0.5\textwidth]{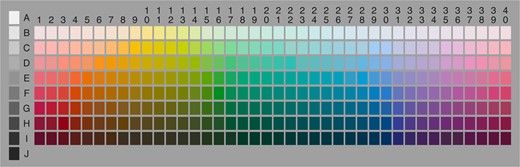}

(chip chart)
\end{center}

\includegraphics[width=\textwidth]{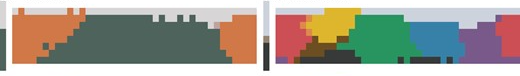}

\hspace{0.25\textwidth}(a) \hspace{0.45\textwidth}(b) 
    \caption{Figures from \cite{zaslavsky2022evolution} depicting the Nafaanra color naming system in 1978 (a) and in 2018 (b) for the colors in the chip chart from the World Color Survey (WCS) used as a stimulus grid. (a) The 1978 system: {\em fi\textipa{N}ge} `light', {\em w\textipa{OO}} `dark', and {\em nyi\textipa{E}} `warm or red-like.' (b) The 2018 system: the three terms from 1978 have smaller expressive ranges, and new terms have emerged---{\em wr\textipa{E}nyi\textipa{N}ge} `green', {\em lomru} `orange', {\em \textipa{N}gonyina} `yellow-orange', {\em mbruku} `blue', {\em poto} `purple', {\em wr\textipa{E}waa} `brown', and {\em t\textipa{OO}nr\textipa{O}} `gray'. }
    \label{fig:nafaanra-color}
\end{figure}

In its evolution, the 2018 system redeployed terms from other systems. The emerging system has borrowed some words from other languages; for example, Nafaanra {\em mbruku} may have its origin in English {\em blue}.   In other cases, the emerging system has redeployed words from other Nafaanra domains, such as {\em \textipa{N}gonyina} ‘yellow-orange’ which comes from the Nafaanra word for chicken fat. The language has repurposed terms from other systems in order to increase the number of distinctions within the color term system. (This is another illustration of the critical interaction of systems of distinction in language mentioned at the start of the section.) Because it makes both more distinctions and distinctions based on other features, the 2018 system has more precise expressive capacity than the 1978 system had.

\end{example}

\section{Summary}\label{sec:discussion}
This article discusses a connection between the Fundamental Theorem of Galois Theory and the Stone--Weierstrass theorem. Intuitively, both theorems state a correspondence between \emph{distinguishing} and \emph{expressing}. Here, we mathematically formalize these concepts and how they relate. We also describe  applied contexts in which these ideas appear in machine learning and data science. Finally, we illustrate this principle through examples in linguistics. 

\bibliographystyle{alpha}

\bibliography{sample}

\end{document}